\documentclass[10pt]{amsart}
\usepackage{amsfonts}
\usepackage{ifthen}
\usepackage{amsthm}
\usepackage{amsmath}
\usepackage{graphicx}
\usepackage{amscd,amssymb,amsthm}

\newcounter{minutes}
\divide\time by 60
\newcounter{hours}
\multiply\time by 60 \addtocounter{minutes}{-\time}

\newcommand{\real}{\operatorname{Re}}

\newtheorem{theorem}{Theorem}
\newtheorem{lemma}{Lemma}


\title[\'A. Baricz, R. Sz\'asz/The radius of convexity of normalized Bessel functions]{The radius of convexity of normalized Bessel functions}

\author[]{\'Arp\'ad Baricz}
\address{Department of Economics, Babe\c{s}-Bolyai University, 400591 Cluj-Napoca, Romania}
\address{Institute of Applied Mathematics, \'Obuda University, 1034 Budapest, Hungary}
\email{bariczocsi@yahoo.com}

\author[]{R\'obert Sz\'asz}
\address{Department of Mathematics and Informatics, Sapientia Hungarian University of Transylvania, 540485 T\^argu-Mure\c{s}, Romania}
\email{rszasz@ms.sapientia.ro}

\thanks{$^{\bigstar}$The research of \'A. Baricz was supported by a research grant of the Romanian National Authority for Scientific Research, CNCS-UEFISCDI, project number PN-II-RU-TE-2012-3-0190.}

\begin{document}

\def\thefootnote{}
\footnotetext{ \texttt{File:~\jobname .tex,
          printed: \number\year-0\number\month-\number\day,
          \thehours.\ifnum\theminutes<10{0}\fi\theminutes}
} \makeatletter\def\thefootnote{\@arabic\c@footnote}\makeatother

\maketitle

\begin{abstract}
The radius of convexity of two normalized Bessel functions of the first kind are determined in the case when the order is between $-2$ and $-1.$ Our methods include the minimum principle for harmonic functions, the Hadamard factorization of some Dini functions, properties of the zeros of Dini functions via Lommel polynomials and some inequalities for complex and real numbers.
\end{abstract}

\section{\bf Introduction and the Main Results}

Let $\mathbb{D}(z_0,r)=\{z\in\mathbb{C}:|z-z_0|<r\}$ denote the open disk centered at $z_0$ and of radius $r>0.$ Let $\mathcal{A}$ be the
class of analytic and univalent functions which are defined on the disk $\mathbb{D}(0,r),$ and are normalized by the conditions $f(0)=f'(0)-1=0$. Note that a function $f\in\mathcal{A}$ is convex in $\mathbb{D}(0,r)$ if and only if
$$\real \left(1+\frac{zf''(z)}{f'(z)}\right)>0 \ \mbox{for all}\ z\in \mathbb{D}(0,r).$$
Moreover, we say that $f$ is a convex function of order $\alpha\in[0,1)$ in $\mathbb{D}(0,r)$ if
$$\real \left(1+\frac{zf''(z)}{f'(z)}\right)>\alpha \ \mbox{for all}\ z\in \mathbb{D}(0,r),$$
and the real number
$$r^c_{f}(\alpha)=\sup\left\{r\in(0,\infty):\ \real \left(1+\frac{zf''(z)}{f'(z)}\right) >\alpha,\ \mbox{for all}\ z\in{\mathbb{D}}(0,r)\right\}$$
is called the radius of convexity of order $\alpha$ of the function $f$. The real number $r^{c}_f(0)=r^{c}_f$ is in fact the largest radius for which the image domain $f(\mathbb{D}(0,r^c_f))$ is convex.

The Bessel function of the first kind is defined by the following power series
$$J_{\nu}(z)=\sum_{n\geq0}\frac{(-1)^{n}\left(z/2\right)^{2n+\nu}}{n!\Gamma(n+\nu+1)}.$$
The radii of convexity (and of starlikeness) of the next three kind of normalized Bessel functions
$$f_{\nu}(z)=(2^{\nu}\Gamma(1+\nu)J_{\nu}(z))^{{1}/{\nu}},\nu\neq{0},$$
$$g_{\nu}(z)=2^{\nu}\Gamma(1+\nu)z^{1-\nu}J_{\nu}(z),$$
$$h_{\nu}(z)=2^{\nu}\Gamma(1+\nu)z^{1-\nu/2}J_{\nu}(z^{\frac{1}{2}})$$
were investigated in the papers \cite{bks,bs,br,sz}, see also the references therein. The paper \cite{bs} contains the radius of convexity of the functions $f_{\nu},$ $g_{\nu}$ and $h_{\nu}$ in the case when $\nu>-1.$ In the proof of the main results of \cite{bs} it was essential that the Bessel functions of the first kind have only real zeros and also the fact that some Dini functions of the form $z\mapsto aJ_{\nu}(z)+bzJ'_{\nu}(z)$ have only real zeros too. This time we deal with the convexity of $g_{\nu}$ and $h_{\nu}$ in the case when $\nu\in(-2,-1).$ In this case the function $z\mapsto aJ_{\nu}(z)+bzJ'_{\nu}(z)$ has complex zeros too and this complicates the problem. If $\nu\in(-2,-1),$ then the method which has been used in \cite{bs} is not applicable directly. To face the difficulty we first prove some results on the zeros of Dini functions when $\nu\in(-2,-1).$ We follow the method of Hurwitz and we use the Lommel polynomials, see Lemma \ref{lemroots} and its proof. Since our method in this paper cannot be used for the function $f_{\nu},$ we note that it would of interest to see the radius of convexity of $f_{\nu}$ when $\nu\in(-1,0)$ or $(-2,-1).$

The paper is organized as follows: at the end of this section we present the main results; section 2 contains the preliminary results and their proofs; while section 3 contains the proofs of the main results. It is worth to mention that Lemma \ref{lemroots}, Lemma \ref{lemg} and Lemma \ref{lemh} are the key tools of the main results, but actually they are of independent interest and may be useful in other problems related to Bessel functions and zeros of Bessel functions. Here and in the sequel $I_{\nu}$ denotes the modified Bessel function of the first kind and order $\nu.$

Our main results of the paper are the following theorems.

\begin{theorem}\label{th1}
If $\nu\in(-2,-1)$ and $\alpha\in[0,1),$ then the radius of convexity of order $\alpha$ of the function $g_{\nu}$ is the smallest positive root of the equation
\begin{equation}\label{k2ac7zdl0m}
 1+r\frac{rI_{\nu+2}(r)+3I_{\nu+1}(r)}{I_{\nu}(r)+rI_{\nu+1}(r)}=\alpha.
\end{equation}
 \end{theorem}

\begin{theorem}\label{th2}
If $\nu\in(-2,-1)$ and $\alpha\in[0,1),$ then the radius of convexity of order $\alpha$ of the function $h_{\nu}$ is the smallest positive root of
the equation
\begin{equation}\label{neweq}1+\frac{rI_{\nu+2}(r^{\frac{1}{2}})+4r^{\frac{1}{2}}I_{\nu+1}(r^{\frac{1}{2}})}{4I_{\nu}(r^{\frac{1}{2}})+
2r^{\frac{1}{2}}I_{\nu+1}(r^{\frac{1}{2}})}=\alpha.\end{equation}
\end{theorem}

\section{\bf Preliminaries}
\setcounter{equation}{0}

In order to prove the main results we need the following preliminary results. The first lemma contains some know
results of Hurwitz on Lommel polynomials (see \cite{hu,wa1} and \cite[p. 305]{wa})
$$g_{2m,\nu}(z)=\sum_{n=0}^m(-1)^n\frac{(2m-n)!}{n!(2m-2n)!}\frac{\Gamma(\nu+2m-n+1)}{\Gamma(\nu+n+1)}z^n.$$

\begin{lemma}\label{lemHur}
Let $m,s\in\mathbb{N}$ and $\nu\in\mathbb{R}.$ The following statements are valid:
\begin{enumerate}
\item[\bf a.] If $\nu>-1,$ then the Lommel polynomial $g_{2m,\nu}(z)$ has only positive real zeros.
\item[\bf b.] If $\nu\in(-2s-2,-2s-1),$ then the polynomial $g_{2m,\nu}(z)$ has $2s$ complex zeros and $m-2s$ real zeros. Between
the real zeros a root is negative and $m-2s-1$ are positive.
\item[\bf c.] If $\nu\in(-2s-1,-2s),$ then $g_{2m,\nu}(z)$ has $2s$ complex zeros and $m-2s$ positive real zeros.
\end{enumerate}
\end{lemma}

The next result is very important in the proof of the main results and may be of independent interest.
We note that Lemma \ref{lemroots} complement the result of Spigler \cite{spigler}, who proved that if $\nu>-1$ and $-a/b>\nu,$ then the equation
$aJ_{\nu}(z)+bzJ_{\nu}'(z)=0$ has an infinite sequence of positive roots and two purely imaginary roots.

\begin{lemma}\label{lemroots}
If $\alpha\geq0,$ then the following assertions are true:
\begin{enumerate}
\item[\bf a.] If $\nu>-1,$ then all the zeros of the equation $J_\nu(z)+\alpha{z}J'_\nu(z)=0$ are real.
\item[\bf b.] If  $\nu\in(-2,-1),$ then the function $z\mapsto J_\nu(z)+\alpha{z}J'_\nu(z)$ has two purely imaginary zeros and all the other zeros are real.
\end{enumerate}
\end{lemma}

\begin{proof}[\bf Proof]
We have to discuss the case $\alpha>0,$ as the case $\alpha=0$ is well-known due to Hurwitz's result on zeros of Bessel functions of the first kind, see \cite{hu}. Let $g_{m,\nu}$ be the $m$-th Lommel polynomial. According to Lemma \ref{lemHur} if $\nu>-1,$ then $g_{2m,\nu}(z)$ has only positive real zeros; and if $\nu\in(-2,-1),$ then $g_{2m,\nu}(z)$ has only real zeros, of which exactly one is negative and all the others are positive. We consider the function  $\omega:\mathbb{C}\rightarrow\mathbb{R},$ defined by $\omega(z)=z^\frac{1}{\alpha}g_{2m,\nu}(z).$ Since $\omega(0)=0,$ the Rolle theorem implies that the equation $\omega'(z)=0$ has only positive real zeros if $\nu\in(-1,\infty),$ and for $\nu\in(-2,-1)$ has a negative zero and $m-1$ positive different zeros.    The equation  $\omega'(z)=0$ is equivalent to $g_{2m,\nu}(z)+\alpha{z}g_{2m,\nu}'(z)=0.$ Thus, the polynomial $h_{m,\nu}(z)=g_{2m,\nu}(z)+\alpha{z}g_{2m,\nu}'(z)$ of degree $m$ has only positive zeros if $\nu\in(-1,\infty),$ and has exactly one negative zero and $m-1$ positive zeros for $\nu\in(-2,-1).$ If $x_1<x_2<{\dots}<x_m$ denote the zeros of $g_{2m,\nu}(z)$ and $y_1<y_2<{\dots}<y_m$ the zeros of $h_{m,\nu}(z),$ then the following inequalities hold
\begin{equation}\label{f1s0a8d662wehl14}
x_1<y_1<0<y_2<x_2<y_3<x_3<{\dots}<x_{m-1}<y_m<x_m.
\end{equation}
On the other hand, we have \cite[p. 484]{wa}
$$\lim_{m\rightarrow\infty}\frac{h_{m,\nu}(z)}{\Gamma(\nu+2m+1)}=f_\nu(z)+\alpha{z}f_\nu'(z),$$  where  $$f_\nu(z)=2^{\nu}z^{-\nu}J_{\nu}(2\sqrt{z})=\sum_{n\geq0}\frac{(-1)^nz^n}{n!\Gamma(\nu+n+1)},$$
and the convergence is uniform on compact subsets of $\mathbb{C}.$  Thus, it follows that $z\mapsto f_\nu(z)+\alpha{z}f_\nu'(z)$ has exactly one
negative real zero and all the other zeros are real and
positive. This means that if we denote by $\{\pm\alpha_{\nu,n}|n\in\mathbb{N^*}\}$ the set of the zeros of
the equation $J_\nu(z)+\alpha{z}J'_{\nu}(z)=0,$ where $\alpha_{\nu,1}=ia,$ $a>0$ and
$0=\real\alpha_{\nu,1}<\alpha_{\nu,2}<\alpha_{\nu,3}<\ldots,$
then (\ref{f1s0a8d662wehl14}) implies that
$$0<a<b, \ 0<j_{\nu,2}<\alpha_{\nu,2}<j_{\nu,3}<\alpha_{\nu,3}<j_{\nu,4}<\alpha_{\nu,4}<\ldots,$$  where $j_{\nu,1}=ib, \ 0<j_{\nu,2}<j_{\nu,3}<\ldots$ are the zeros of $J_{\nu}(z)=0.$
\end{proof}

The next two lemmas have been proved in \cite{bs} provided that $\nu>-1.$ The main tool in the proofs of these lemmas were the following  estimations: if $H_\nu^{(1)}$ and $H_\nu^{(2)}$ denote the Bessel functions of the third kind then the following asymptotic expansions hold
$$H_\nu^{(1)}(w)=\left(\frac{2}{\pi{w}}\right)^{\frac{1}{2}}e^{i(w-\frac{1}{2}\nu\pi-\frac{1}{4}\pi)}(1+\eta_{1,\nu}(w)),\ H_\nu^{(2)}(w)=\left(\frac{2}{\pi{w}}\right)^{\frac{1}{2}}e^{-i(w-\frac{1}{2}\nu\pi-\frac{1}{4}\pi)}(1+\eta_{2,\nu}(w)),$$
where $\eta_{1,\nu}(w)$ and $\eta_{2,\nu}(w)$ are $\mathcal{O}(1/w),$ when $|w|$ is large. Since these estimations hold for every $\nu\in\mathbb{C}$ (see \cite[p. 198]{wa}) and the condition $\nu>-1$ has not been used, it follows that the following two lemmas hold for every $\nu\in\mathbb{R}.$

\begin{lemma}\label{lemg}
Let $z\in\mathbb{C}$ and let $\alpha_{\nu,n},$ $n\in\mathbb{N},$
denote the $n$-th zero of the equation  $J_\nu(z)-zJ_{\nu+1}(z)=0,$
where $0\leq\real\alpha_{\nu,1}\leq\real\alpha_{\nu,2}\leq{\dots}\leq\real\alpha_{\nu,n}\leq{\dots}.$ The following development holds for every $z\neq{\alpha_{\nu,n}}$
\begin{equation}\label{bbb4nhs4569033}\frac{g''_{\nu}(z)}{g'_{\nu}(z)}=\frac{zJ_{\nu+2}(z)-3J_{\nu+1}(z)}{J_\nu(z)-zJ_{\nu+1}(z)}
=-\sum_{n\geq1}\frac{2z}{\alpha^2_{\nu,n}-z^2}.\end{equation}
\end{lemma}

\begin{lemma}\label{lemh}
Let $z\in\mathbb{C}$ and let $\beta_{\nu,n},$ $n\in\mathbb{N},$ denote the $n$-th zero of the equation $(2-\nu)J_\nu(z)+zJ_\nu'(z)=0,$ where
$0\leq\real\beta_{\nu,1}\leq\real\beta_{\nu,2}\leq{\dots}\leq\real\beta_{\nu,n}\leq{\dots}.$ The following development holds for each $z\neq\beta_{\nu,n}$
\begin{equation}\label{bbb4nhs456x1s2d5f69033}z\frac{h_{\nu}''(z)}{h_{\nu}'(z)}=
\frac{\nu(\nu-2)J_\nu(z^{\frac{1}{2}})+(3-2\nu)z^{\frac{1}{2}}J_\nu'(z^{\frac{1}{2}})+
zJ_\nu''(z^{\frac{1}{2}})}{2(2-\nu)J_\nu(z^{\frac{1}{2}})+2z^{\frac{1}{2}}J_\nu'(z^{\frac{1}{2}})}=-\sum_{n\geq1}\frac{z}{\beta^2_{\nu,n}-z}.\end{equation}
\end{lemma}

To prove the main results we will need also the next results.

\begin{lemma}
If $v\in{\mathbb{C}},$ $\delta\in{\mathbb{R}}$ and $\delta>|v|,$ then
\begin{equation}\label{a}  \frac{|v|}{\delta-|v|}\geq {\real}\left(\frac{v}{\delta-v}\right)\geq\frac{-|v|}{\delta+|v|},\end{equation}
\begin{equation}\ \label{abc2s4d}
\frac{|v|}{\delta+|v|}\geq\real\left(\frac{v}{\delta+v}\right)\geq\frac{-|v|}{\delta-|v|}.\end{equation}
Moreover, if $v\in{\mathbb{C}},$ $\gamma,\delta\in{\mathbb{R}}$ and $\gamma\geq\delta>r\geq|v|,$ then
\begin{equation}\label{4l5h6ddg89u}\frac{r^2}{(\delta-r)(\gamma+r)}\geq \real\left(\frac{v^2}{(\delta+v)(\gamma-v)}\right).\end{equation}
\end{lemma}

\begin{proof}[\bf Proof]
In order to prove (\ref{a}), we denote $v=re^{i\theta}$ and we consider the function $u:[0,2\pi]\rightarrow\mathbb{R},$ defined by
$$u(\theta)=\real\frac{v}{\delta-v}=\frac{r\delta\cos\theta-r^2}{\delta^2+r^2-2r\delta\cos\theta}.$$
Since $$u'(\theta)=\frac{\delta{r}(r^2-\delta^2)\sin\theta}{(\delta^2+r^2-2r\delta\cos\theta)^2},$$
the function $u$ is decreasing on $[0,\pi]$ and increasing on $[\pi,2\pi],$ which implies that
$$\frac{|v|}{\delta-|v|}=u(0)=u(2\pi)\geq \real\left(\frac{v}{\delta-v}\right)\geq{u(\pi)}=\frac{-|v|}{\delta+|v|}.$$
The inequality (\ref{abc2s4d}) can be proved in a similar way. Now, in order to
prove (\ref{4l5h6ddg89u}) we note that it is enough to prove this inequality in the case $v=re^{i\theta},$ $\theta\in[0,2\pi].$ Thus, we have to show that
$$w(\pi)\geq{w(\theta)},  \  \textrm{where} \   w(\theta)=\frac{\gamma\delta\cos2\theta+r(\gamma-\delta)\cos\theta-r^2}{(\gamma^2+r^2-2\gamma{r}\cos\theta)(\delta^2+r^2+2\delta{r}\cos\theta)},\
\theta\in[0,2\pi].$$ Denoting $t=\cos\theta$  we obtain that this inequality is equivalent to $w_1(-1)\geq{w_1(t)},$  where $t\in[-1,1]$ and $$w(\theta)=w_1(t)=\frac{2\gamma\delta{t^2}+r(\gamma-\delta)t-r^2-\gamma\delta}{(\gamma^2+r^2-2\gamma{r}t)(\delta^2+r^2+2\delta{r}t)},$$
and this can be rewritten as $w_2(t)\geq0,$ where $t\in[-1,1]$ and
$$w_2(t)=(\gamma^2+r^2-2\gamma{r}t)(\delta^2+r^2+2\delta{r}t)- (2\gamma\delta{t^2}+r(\gamma-\delta)t-r^2-\gamma\delta)(\gamma+r)(\delta-r).$$
We note that $w_2(t)$ is a polynomial of degree two, and its roots satisfy $t_1=-1$ and
$$t_1t_2=-\frac{(\gamma^2+r^2)(\delta^2+r^2)+(r^2+\gamma\delta)(\gamma+r)(\delta-r)}{4\gamma\delta{r^2}+2\gamma\delta(\gamma+r)(\delta-r)}.$$
Thus, we get
$$t_2=\frac{(\gamma^2+r^2)(\delta^2+r^2)+(r^2+\gamma\delta)(\gamma+r)(\delta-r)}{4\gamma\delta{r^2}+2\gamma\delta(\gamma+r)(\delta-r)}.$$
Consequently, the inequality $w_2(t)\geq0$ is equivalent to $(1+t)(t_2-t)\geq0,$ $t\in[-1,1].$ In order to finish the proof we have to show that $t_2\geq1.$
A short calculation shows that this inequality is equivalent to
$$r(\gamma\delta-r^2)(\gamma-\delta)+r^2(\gamma-\delta)^2\geq0,$$
and with this the proof is done.
\end{proof}

\section{\bf Proofs of the Main Results}
\setcounter{equation}{0}

\begin{proof}[\bf Proof of Theorem \ref{th1}]
By using
$$z\frac{g_{\nu}''(z)}{g_{\nu}'(z)}=\frac{\nu(\nu-1)J_\nu(z)+2(1-\nu)zJ_\nu'(z)+z^2J_\nu''(z)}{(1-\nu)J_\nu(z)+zJ_\nu'(z)},$$
the fact that $J_{\nu}$ is a particular solution of the Bessel differential equation,
and the recurrence formula $zJ'_\nu(z)=\nu{J_\nu(z)}-zJ_{\nu+1}(z),$ it follows that
$$z\frac{g''_{\nu}(z)}{g'_{\nu}(z)}=z\frac{zJ_{\nu+2}(z)-3J_{\nu+1}(z)}{J_{\nu}(z)-zJ_{\nu+1}(z)}.$$
In view of \eqref{bbb4nhs4569033} we obtain
$$1+z\frac{g''_{\nu}(z)}{g'_{\nu}(z)}=1-2\sum_{n\geq1}\frac{z^2}{\alpha_{\nu,n}^2-z^2}.$$
By using Lemma \ref{lemroots} for $\alpha=(1-\nu)^{-1},$ the condition $\nu\in(-2,-1)$ implies $\alpha_{\nu,1}=ia,$ $a>0$  and   $\alpha_{\nu,n}>0$ for $n\in\{2,3,\dots\}.$ Thus, we get
$$1+z\frac{g''_{\nu}(z)}{g'_{\nu}(z)}=1+\frac{2z^2}{a^2+z^2}-2\sum_{n\geq2}\frac{z^2}{\alpha_{\nu,n}^2-z^2}
=1+\frac{1}{a^2}\frac{2z^2}{1+\frac{z^2}{a^2}}-2\sum_{n\geq2}\frac{z^2}{\alpha_{\nu,n}^2-z^2}.$$
On the other hand, the convergence of the function series in (\ref{bbb4nhs4569033}) is uniform on every compact subset of
$\mathbb{C}\setminus\{\alpha_{\nu,n}|n\in\mathbb{N}\}.$ Integrating both sides of the equality (\ref{bbb4nhs4569033}), it follows that
\begin{eqnarray}\label{bbb4ngkmhs4569mnd033}{g'_{\nu}(z)}=\prod_{n\geq1}\left(1-\frac{z^2}{\alpha^2_{\nu,n}}\right).\end{eqnarray}
The convergence of the infinite product is uniform on every compact subset of $\mathbb{C}.$ Comparing the coefficients of $z^2$ on both sides of
(\ref{bbb4ngkmhs4569mnd033}) we get the following  equality
\begin{equation}\label{ed15s4z7ax4x4x}
\sum_{n\geq1}\frac{1}{\alpha^2_{\nu,n}}=\frac{3}{4(\nu+1)}.
\end{equation}
The equality  (\ref{ed15s4z7ax4x4x}) implies
$$\frac{1}{a^2}=-\frac{3}{4(\nu+1)}+\sum_{n\geq2}\frac{1}{\alpha_{\nu,n}^2}$$
and using this we obtain that
$$1+z\frac{g''_{\nu}(z)}{g'_{\nu}(z)}=1-\frac{3a^2}{2(\nu+1)}\frac{z^2}{a^2+z^2}-
2\sum_{n\geq2}\frac{\alpha_{\nu,n}^2+a^2}{\alpha_{\nu,n}^2}\frac{z^4}{(\alpha_{\nu,n}^2-z^2)(a^2+z^2)}.$$
On the other hand, we have $-\frac{3a^2}{2(\nu+1)}>0,$ and taking
$v=z^2$ in the inequality (\ref{abc2s4d}) we get
$$\real\frac{z^2}{a^2+z^2}\geq\frac{-|z|^2}{a^2-|z|^2}\geq\frac{-r^2}{a^2-r^2},$$ for all $|z|\leq{r}<a.$ Moreover, taking $v=z^2$ in inequality    (\ref{4l5h6ddg89u}) it follows that $$\real\frac{z^4}{(\alpha_{\nu,n}^2-z^2)(a^2+z^2)}\leq\frac{r^4}{(\alpha_{\nu,n}^2+r^2)(a^2-r^2)},$$  for all   $|z|\leq{r}<a<\alpha_{\nu,n}$ and $n\in\{2,3,\dots\}.$ Summarizing, provided that $|z|\leq{r}<a,$ the following
inequality holds
$$\real\left(1+z\frac{g''_{\nu}(z)}{g'_{\nu}(z)}\right)\geq1+\frac{3a^2}{2(\nu+1)}\frac{r^2}{a^2-r^2}-
2\sum_{n\geq2}\frac{\alpha_{\nu,n}^2+a^2}{\alpha_{\nu,n}^2}\frac{r^4}{(\alpha_{\nu,n}^2+r^2)(a^2-r^2)}=\left.
\left[\real\left(1+z\frac{g''_{\nu}(z)}{g'_{\nu}(z)}\right)\right]\right|_{z=ir}.$$
This means that
$${\inf_{z\in\mathbb{D}(0,r)}}{\real}\left(1+z\frac{g''_{\nu}(z)}{g'_{\nu}(z)}\right)=
1+ir\frac{g''_{\nu}(ir)}{g'_{\nu}(ir)}=1-\frac{2r^2}{a^2-r^2}+2\sum_{n\geq2}\frac{r^2}{\alpha_{\nu,n}^2+r^2},$$
for all $r\in(0,a).$ Now, let us consider the function
$\varphi:(0,a)\rightarrow\mathbb{R},$  defined by
$$\varphi(r)=1-\frac{2r^2}{a^2-r^2}+2\sum_{n\geq 2}\frac{r^2}{\alpha_{\nu,n}^2+r^2}=1+ir\frac{g''_{\nu}(ir)}{g'_{\nu}(ir)}.$$ This function satisfy
$\lim_{r\searrow0}\varphi(r)=1>\alpha, \
\lim_{r\nearrow{a}}\varphi(r)=-\infty,$ and
$$\varphi'(r)=-\frac{4ra^2}{(a^2-r^2)^2}+\sum_{n\geq 2}\frac{4r\alpha_{\nu,n}^2}{(\alpha_{\nu,n}^2+r^2)^2}<
-\frac{4ra^2}{(a^2-r^2)^2}+\sum_{n\geq 2}\frac{4r}{\alpha_{\nu,n}^2}=-\frac{4ra^2}{(a^2-r^2)^2}+\frac{4r}{a^2}+\frac{3r}{\nu+1}<\frac{3r}{\nu+1}.$$
In other words, the function $\varphi$ maps $(0,a)$ into $(-\infty,1)$ and is strictly decreasing. Thus the equation
$1+ir\frac{g''_{\nu}(ir)}{g'_{\nu}(ir)}=\alpha$ has exactly one root in the interval $(0,a),$ and this equation is equivalent to
(\ref{k2ac7zdl0m}). If we denote by $r_2\in(0,a)$ the unique root of the equation
$1+ir\frac{g''_{\nu}(ir)}{g'_{\nu}(ir)}=\alpha,$ then by using the minimum principle of harmonic functions we have that
$$\real\left(1+\frac{zg_{\nu}''(z)}{g_{\nu}'(z)}\right)>\alpha\ \ \ \mbox{for all}\ \ \ z\in{\mathbb{D}(0,r_2)},$$
$$\inf_{{z\in{\mathbb{D}(0,r_2)}}}\real\left(1+\frac{zg_{\nu}''(z)}{g_{\nu}'(z)}\right)=\alpha,\ \ \inf_{{z\in{\mathbb{D}(0,r)}}}\real\left(1+\frac{zg_{\nu}''(z)}{g_{\nu}'(z)}\right)<\alpha,
\ \ \textrm{for} \ \ r>r_2,$$
 and the proof is done.
\end{proof}

\begin{proof}[\bf Proof of Theorem \ref{th2}]
Lemma \ref{lemh} implies that
$$z\frac{h_{\nu}''(z)}{h_{\nu}'(z)}=\frac{\nu(\nu-2)J_\nu(z^{\frac{1}{2}})+(3-2\nu)z^{\frac{1}{2}}J_\nu'(z^{\frac{1}{2}})+
zJ_\nu''(z^{\frac{1}{2}})}{2(2-\nu)J_\nu(z^{\frac{1}{2}})+2z^{\frac{1}{2}}J_\nu'(z^{\frac{1}{2}})}=-\sum_{n\geq1}\frac{z}{\beta^2_{\nu,n}-z}.$$
By using Lemma \ref{lemroots} for $\alpha=(2-\nu)^{-1},$ if $\nu\in(-2,-1),$ then $\beta_{1,\nu}=ic$ and $0<\beta_{\nu,2}<\beta_{\nu,3}<{\dots}<\beta_{\nu,n}<{\dots},$ and we infer
$$1+z\frac{h_{\nu}''(z)}{h_{\nu}'(z)}=1+\frac{z}{c^2+z}-\sum_{n\geq2}\frac{z}{\beta^2_{\nu,n}-z}.$$
On the other hand, by integrating both sides of the equality (\ref{bbb4nhs456x1s2d5f69033}), it follows that
\begin{equation}\label{bbb4nhg45d1s456x1s2d5f69033}h_\nu(z)=\prod_{n\geq1}\left(1-\frac{z}{\beta^2_{\nu,n}}\right).\end{equation}
Comparing the coefficients of $z$ on both sides of (\ref{bbb4nhg45d1s456x1s2d5f69033}) we get the following  equality
$$\sum_{n\geq1}\frac{1}{\beta^2_{\nu,n}}=\frac{1}{2(\nu+1)},$$
which in turn implies that
$$-\frac{1}{c^2}+\sum_{n\geq 2}\frac{1}{\beta^2_{\nu,n}}=\frac{1}{2(\nu+1)}.$$
Thus, we get
\begin{equation}\label{x1c28f4sj146669}1+z\frac{h_{\nu}''(z)}{h_{\nu}'(z)}=
1-\frac{c^2}{2(\nu+1)}\frac{z}{c^2+z}+\sum_{n\geq2}\left[-\frac{c^2+\beta^2_{\nu,n}}{\beta^2_{\nu,n}}
\frac{z^2}{(c^2+z)(\beta^2_{\nu,n}-z)}\right].\end{equation}
If $c^2>r\geq|z|,$ then the second inequality of (\ref{abc2s4d}) implies that
$$\real\left[-\frac{c^2}{2(\nu+1)}\frac{z}{c^2+z}\right]\geq\frac{c^2}{2(\nu+1)}\frac{r}{c^2-r}.$$
Since for $n\in\{2,3,\dots\}$ we have $\beta_{\nu,n}^2>c^2>r\geq|z|,$ by using the inequality (\ref{4l5h6ddg89u}) it follows that
$$\real\left[-\frac{c^2+\beta^2_{\nu,n}}{\beta^2_{\nu,n}}\frac{z^2}{(c^2+z)(\beta^2_{\nu,n}-z)}\right]
\geq-\frac{c^2+\beta^2_{\nu,n}}{\beta^2_{\nu,n}}\frac{r^2}{(c^2-r)(\beta^2_{\nu,n}+r)}.$$
If $c^2>r\geq|z|,$  then  these two inequalities and (\ref{x1c28f4sj146669}) together imply that
$$\real\left[1+z\frac{h_{\nu}''(z)}{h_{\nu}'(z)}\right]
\geq1+\frac{c^2}{2(\nu+1)}\frac{r}{c^2-r}-\sum_{n\geq2}\frac{c^2+\beta^2_{\nu,n}}{\beta^2_{\nu,n}}
\frac{r^2}{(c^2-r)(\beta^2_{\nu,n}+r)}=1-r\frac{h_{\nu}''(-r)}{h_{\nu}'(-r)}.$$
Thus, we obtain
$${\inf_{z\in\mathbb{D}(0,r)}}\real\left(1+z\frac{h''_{\nu}(z)}{h'_{\nu}(z)}\right)=1-r\frac{h_{\nu}''(-r)}{h_{\nu}'(-r)} \ \textrm{for \ all} \  \ r\in(0,c^2).$$
Now, consider the function $\phi:(0,c^2)\rightarrow\mathbb{R},$ defined by
$$\phi(r)=1+\frac{c^2}{2(\nu+1)}\frac{r}{c^2-r}-\sum_{n\geq2}\frac{c^2+\beta^2_{\nu,n}}{\beta^2_{\nu,n}}\frac{r^2}{(c^2-r)(\beta^2_{\nu,n}+r)}
=1-r\frac{h_{\nu}''(-r)}{h_{\nu}'(-r)}.$$
Since
$$\lim_{r\searrow0}\phi(r)=1, \ \lim_{r\nearrow{c^2}}\phi(r)=-\infty,$$
it follows that the equation $1-r\frac{h_{\nu}''(-r)}{h_{\nu}'(-r)}=\alpha$  has at last one real root in the interval $(0,c^2).$
Let $r_3$  denote the smallest positive real root of the equation $1-r\frac{h_{\nu}''(-r)}{h_{\nu}'(-r)}=\alpha.$
We have
$$\real\left(1+\frac{zh_{\nu}''(z)}{h_{\nu}'(z)}\right)>\alpha, \
z\in{\mathbb{D}(0,r_3)},$$ and if $r>r_3$ then
$$\inf_{z\in{\mathbb{D}(0,r)}}\real\left(1+\frac{zh_{\nu}''(z_0)}{h_{\nu}'(z)}\right)<\alpha.$$
In order to finish the proof we remark that
$1-r\frac{h_{\nu}''(-r)}{h_{\nu}'(-r)}=\alpha$  is equivalent to \eqref{neweq}.
\end{proof}

\end{document}